\documentclass[preprint,12pt]{elsarticle}
\usepackage{amssymb}
\usepackage{newunicodechar}
\newunicodechar{‎}{}  
\usepackage{amsmath}
\usepackage{amsthm,float,epstopdf,latexsym}
\usepackage{pgfplots}
\usepackage{tikz,graphicx}
\usepackage{comment}
\usepackage{pgf,tikz}
\usepackage{tikz}
\usepackage{caption}
\usepackage{subcaption} % <-- add this for subfigures
\usetikzlibrary{positioning}
\usetikzlibrary{positioning, shapes.geometric}
\usepackage{caption}
\usepackage{geometry}
\usepackage{setspace}
\usetikzlibrary{arrows}
\newtheorem{thm}{Theorem}[section]
\newtheorem{prop}[thm]{Proposition}
\newtheorem{dfn}[thm]{Definition}

\newtheorem{lem}[thm]{Lemma}
\newtheorem{cor}[thm]{Corollory}

\journal{Nuclear Physics B}

\begin{document}

\begin{frontmatter}



\title{Hardness and Structural Properties of Fuzzy Edge Contraction} %% Article title

\author[inst1]{Shanookha Ali\corref{cor1}}
%\author[inst1]{Maneesha}

%% Affiliations
\affiliation[inst1]{% Department of General Science, Birla Institute of Technology \& Science, Pilani, Dubai Campus, Dubai 345055, United Arab Emirates
  organization={ Department of General Science, Birla Institute of Technology \& Science, Pilani, Dubai Campus}, 
  city={Dubai},
    postcode={345055},
  country={United Arab Emirates}
}

%% Corresponding author
\cortext[cor1]{Corresponding author. Email: shanookha@dubai.bits-pilani.ac.in}
%\author{shanookha} %% Author name

%\affiliation{organization={},%Department and Organization      addressline={},    city={},postcode={},  state={},country={}}

%% Abstract
\begin{abstract}
We investigate the computational complexity of edge-deletion and edge-contraction problems in fuzzy graphs. For any graph property $\Pi$ that is hereditary under contractions (or deletions) and determined by 3-connected components, the corresponding fuzzy edge-deletion (FPED) and fuzzy edge-contraction (FPEC) problems are NP-hard. Our results hold under both fixed-threshold ($\alpha_0$) and all-threshold ($\forall \alpha$) semantics, and apply even to restricted classes of fuzzy graphs such as fuzzy 3-connected or fuzzy bipartite graphs. We further demonstrate that well-known properties, including planarity and series–parallelness, satisfy these conditions, making the fuzzy versions of these classical graph problems computationally intractable. The proofs leverage reductions from classical NP-hard problems and generalize the constructions to the fuzzy setting while preserving key structural properties.
\end{abstract}

%bbb

\begin{comment}
    \begin{highlights}
\item Developed a rigorous framework for fuzzy edge contraction via $t$-norms.

\item Proved contraction commutes with $\alpha$-cuts, preserving crisp graph properties.

\item Established monotonicity, order-independence, and hereditary preservation.

\item Showed Fuzzy PEC($\sim$) is NP-hard via reduction from connected node cover.

\item NP-hardness holds even for fuzzy 3-connected and fuzzy bipartite graphs.
\end{highlights}
\end{comment}

%% Keywords

\end{frontmatter}

\section{Introduction}

\noindent Edge-deletion and edge-contraction are two of the most fundamental 
operations in graph theory, and they play a central role in the 
structural study of graphs. Classical results of Asano and Hirata~\cite{AsanoHirata1982, AsanoHirata1983} established that the corresponding 
edge-deletion problem (PED) and edge-contraction problem (PEC) are 
NP-hard for a wide class of hereditary graph properties determined by 
3-connected components. These results form the foundation for much of 
the hardness landscape in graph modification problems.

\noindent In parallel, fuzzy set theory and fuzzy graph theory provide a natural 
framework to model uncertainty in networks. Zadeh~\cite{zadeh1965fuzzy}
introduced the notion of \emph{$\alpha$-cuts} in fuzzy sets, and 
Rosenfeld~\cite{rosenfeld1977fuzzy} extended this idea to fuzzy graphs, where 
$\alpha$-cut graphs provide a crisp approximation of fuzzy structures at 
different thresholds. A comprehensive treatment of fuzzy graphs and their 
operations, including $\alpha$-cut based interpretations, can be found in 
the monograph by Mordeson and Nair~\cite{mordeson2000fuzzy}. In related work, Ramya and Lavanya~\cite{ramya2023contraction} introduced edge contraction and neighbourhood contraction operations for fuzzy graphs and investigated their impact on domination parameters. Their study provides an initial formal framework for contraction-based operations in fuzzy graph theory.\medskip

\noindent Building on this foundation, we introduce and study the 
\emph{fuzzy edge-deletion} (FPED) and \emph{fuzzy edge-contraction} 
(FPEC) problems. Specifically, we formalize two types of 
\emph{$\alpha$-semantics}: the \emph{threshold semantics}, where a fixed 
$\alpha_0$ is considered, and the stronger \emph{all-$\alpha$ semantics}, 
where properties must hold across all thresholds. 
To the best of our knowledge, these formulations do not appear in the 
existing fuzzy graph literature. Our main contribution is to show that 
the NP-hardness of PED and PEC extends robustly to the fuzzy setting, 
under both threshold and all-$\alpha$ semantics, even for structured 
subclasses such as fuzzy 3-connected and bipartite graphs. This work complements and extends earlier research on node connectivity, Hamiltonicity, and container structures in fuzzy graphs with applications to human trafficking~\cite{ali2018vertex,ali2021hamiltonian,ali2024containers}.

\medskip
\noindent In this paper we develop the theory of \emph{fuzzy edge contractions}
under $\alpha$-semantics.  
Formally, contracting an edge $e=uv$ merges $u$ and $v$ into a new node $w$,
and assigns new edge memberships to $w$ using a suitable $t$-norm $T$.
We establish several foundational results form \cite{samanta2011fuzzy}:

\begin{enumerate}
    \item Contraction \emph{commutes with $\alpha$-cuts} for edges
          of membership at least $\alpha$, ensuring that reasoning can be
          carried out entirely in the crisp $\alpha$-cut.
    \item Edge memberships under contraction are \emph{monotone} with respect
          to the original, so $\alpha$-adjacency sets are preserved as intersections
          of neighborhoods.
    \item Contractions on disjoint edges are \emph{order-independent},
          paralleling the crisp case.
    \item Classical hereditary properties determined by $3$-connected components
          \emph{lift naturally} to the fuzzy setting via $\alpha$-cuts.
    \item Connectivity, neighborhood structure, and planarity-type properties
          remain robust under fuzzy contractions, allowing NP-hardness reductions
          from the crisp case.
\end{enumerate}

\noindent Taken together, these results provide the first systematic extension of
contraction-based structural graph theory into the fuzzy setting.
They also form the basis for our complexity results on fuzzy edge-deletion and
edge-contraction problems (FPED and FPEC), under both threshold and all-$\alpha$
semantics, which we show remain NP-hard.
\medskip

\noindent Section~\ref{sec:prelim} recalls preliminaries on fuzzy graphs and $\alpha$-semantics.
Section~\ref{sec:contraction} develops the structural theory of fuzzy edge contraction.
Section~\ref{sec:hardness} applies these tools to complexity questions,
showing NP-hardness for FPED and FPEC in several natural subclasses.
The results establish that with open directions in Section~\ref{sec:conclusion}.

\section{Preliminaries}\label{sec:prelim}

\noindent A fuzzy graph is a pair 
$G_f = (\mu_V,\mu_E)$ 
where $\mu_V: V \to [0,1]$ assigns a membership to each node, 
and $\mu_E: V \times V \to [0,1]$ assigns a membership to each edge, 
with $\mu_E(u,v) \le \min\{\mu_V(u),\mu_V(v)\}$.

For $\alpha \in (0,1]$, the \emph{$\alpha$-cut} of $G_f$ is the crisp graph
\[
G_f^{(\alpha)} = \big(V_\alpha, \mu_E^{\alpha}\big),
\quad
V_\alpha = \{ v \in V : \mu_V(v) \ge \alpha \}, 
\quad
\mu_E^{\alpha} = \{ (u,v) : \mu_E(u,v) \ge \alpha \}.
\]

\begin{dfn}\cite{ramya2023contraction,mordeson2012fuzzy}
Let $G_f = (\mu_V,\mu_E^{~})$ be a fuzzy graph and let $(u,v)$ be an edge of $G_f$ with membership $\mu = \tilde{E}(u,v)$. 
The \emph{fuzzy edge contraction} of $(u,v)$ merges the nodes$u$ and $v$ into a single node $w$, forming a new fuzzy graph $G_f^{uv}$ with
\[
V' = (V \setminus \{u,v\}) \cup \{w\}.
\]

The membership of the new node is 
\[
\mu(w) = \min\{\mu(u), \mu(v)\} \quad \text{(or using an appropriate $t$-norm)}.
\]

For every $x \in V \setminus \{u,v\}$, the edge memberships are updated by
\[
\tilde{E}'(w,x) = T(\tilde{E}(u,x), \tilde{E}(v,x)),
\]
where $T$ is a $t$-norm (commonly the minimum). 
All other edge memberships remain unchanged.
\end{dfn}
\noindent  Figure \ref{fig:your-label} shows fuzzy graph $G_f$ with edge membership values and the resulting $\alpha$-cut graphs. Known results on crisp contractions \cite{AsanoHirata1983}:
\begin{figure}[h!]
\centering

% Style for nodes and edges
\tikzstyle{vertex}=[circle, draw=black, fill=cyan!30, minimum size=15pt, inner sep=0pt]
\tikzstyle{edge}=[line width=0.8pt]
\tikzstyle{boldedge}=[line width=1.2pt]

% ---------- First: Fuzzy Graph --------------
\begin{subfigure}{0.32\textwidth}
\centering
\begin{tikzpicture}[scale=1]

\node[vertex] (1) {1};
\node[vertex, below left=1.2cm and 0.3cm of 1] (2) {2};
\node[vertex, below right=1.2cm and 0.3cm of 2] (5) {5};
\node[vertex, below left=1.2cm and 0.6cm of 2] (3) {3};
\node[vertex, left=1.2cm of 3] (4) {4};

\draw[edge] (1) -- (2) node[midway, above] {0.9};
\draw[edge] (1) -- (5) node[midway, right] {0.6};
\draw[edge, dotted] (4) to[bend left=25] node[midway, below] {0.4} (5);

%\draw[edge] (4) -- (5) node[midway, right] {0.3};
\draw[edge] (2) -- (5) node[midway, above] {0.8};
\draw[edge] (2) -- (3) node[midway, left] {0.3};
\draw[edge] (3) -- (4) node[midway, above] {0.5};
\draw[edge] (2) -- (3) ++(0,-0.1); % shift label
\draw[edge] (2) -- (3) node[midway, right] {0.7};

\end{tikzpicture}
\caption{Fuzzy Graph $G_f$}
\end{subfigure}
%
% ---------- Second: alpha=0.5 --------------
\begin{subfigure}{0.32\textwidth}
\centering
\begin{tikzpicture}[scale=1]

\node[vertex] (1) {1};
\node[vertex, below left=1.2cm and 0.3cm of 1] (2) {2};
\node[vertex, below right=1.2cm and 0.3cm of 2] (5) {5};
\node[vertex, below left=1.2cm and 0.6cm of 2] (3) {3};
\node[vertex, left=1.2cm of 3] (4) {4};

\draw[boldedge] (1) -- (2) node[midway, above] {0.9};
\draw[edge, dotted] (4) to[bend left=25] node[midway, below] {0.4} (5);

\draw[boldedge] (1) -- (5) node[midway, right] {0.6};
\draw[boldedge] (2) -- (5) node[midway, above] {0.8};
\draw[boldedge] (2) -- (3) node[midway, left] {0.3};
\draw[boldedge] (3) -- (4) node[midway, above] {0.5};
\draw[boldedge] (2) -- (3) node[midway, right] {0.7};

\end{tikzpicture}
\caption{$\alpha$-cut $G_f^{(0.5)}$}
\end{subfigure}
%
% ---------- Third: alpha=0.8 --------------
\begin{subfigure}{0.32\textwidth}
\centering
\begin{tikzpicture}[scale=1]

\node[vertex] (1) {1};
\node[vertex, below left=1.2cm and 0.3cm of 1] (2) {2};
\node[vertex, below right=1.2cm and 0.3cm of 2] (5) {5};
\node[vertex, below left=1.2cm and 0.6cm of 2] (3) {3};
\node[vertex, left=1.2cm of 3] (4) {4};

\draw[boldedge] (1) -- (2) node[midway, above] {0.9};
\draw[boldedge] (1) -- (5) node[midway, right] {0.6};
\draw[edge, dotted] (4) to[bend left=25] node[midway, below] {0.4} (5);

\draw[boldedge] (2) -- (5) node[midway, above] {0.8};
\draw[edge] (2) -- (3) node[midway, left] {0.3};
\draw[edge] (3) -- (4) node[midway, above] {0.5};
\draw[edge] (2) -- (3) node[midway, right] {0.7};

\end{tikzpicture}
\caption{$\alpha$-cut $G_f^{(0.8)}$}
\end{subfigure}
\caption{Fuzzy graph $G_f$ showing edge membership values and the resulting $\alpha$-cut graphs for $\alpha=0.5$ and $\alpha=0.8$. Bold edges indicate edges present in the corresponding $\alpha$-cut.}
\label{fig:your-label}
\end{figure}
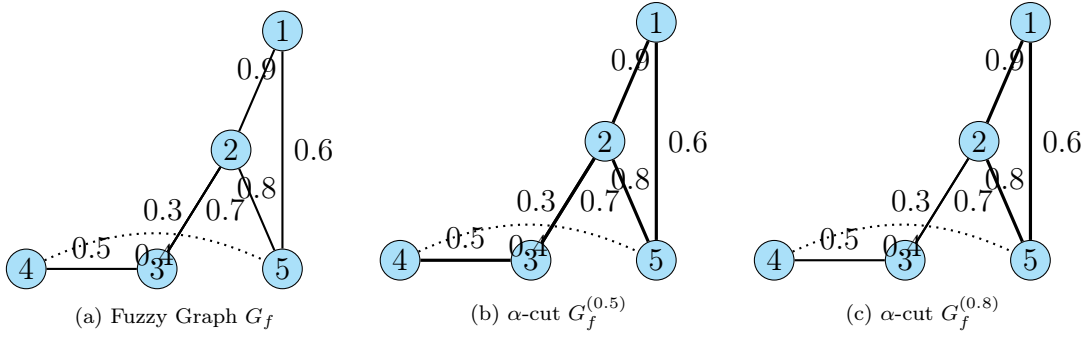

\begin{comment}

\begin{figure}[h!]
  \centering
  \includegraphics[width=0.8\textwidth]{u4czd6nh.png}
  \caption{Description of the figure.}
  \label{fig:your-label}
\end{figure}
\end{comment}

\noindent Two interpretations of graph properties for fuzzy graphs are common:

\begin{itemize}
  \item \textbf{Threshold semantics:} fix $\alpha_0\in(0,1]$ and 
        require the property to hold in $G_f^{(\alpha_0)}$.
  \item \textbf{All-$\alpha$ semantics:} require the property to hold 
        in every $G_f^{(\alpha)}$ for $\alpha \in (0,1]$.
\end{itemize} 
Asano and Hirata~\cite{AsanoHirata1983} proved that the 
\emph{edge-deletion problem} and the \emph{edge-contraction problem} 
are NP-hard for any graph property $\Pi$ that is hereditary under the 
respective operation and is determined by $3$-connected components 
(or biconnected components). 
This foundational result motivates our extension to fuzzy graphs.

\begin{dfn}\cite{mathew2018fuzzy}
A fuzzy graph is \emph{fuzzy 3-connected} if every 2-node cut has total membership below a threshold $\alpha$.
\end{dfn}

\begin{dfn} \cite{bowen1992fuzzy}
A fuzzy graph property $\sim$ is \emph{hereditary under fuzzy contractions} if contracting any fuzzy edge does not create a violation of $\sim$ in any thresholded subgraph.
\end{dfn}

\begin{dfn} \cite{bowen1992fuzzy}
Let $G_f = ( \mu_V,  \mu_E)$ be a fuzzy graph, where 
$\mu_V : V \to [0,1]$ and $\mu_E : E \to [0,1]$ denote the vertex and edge 
membership functions, respectively. For any threshold $\alpha \in (0,1]$, 
the \emph{$\alpha$-cut} of $G_f$ is the crisp graph
\[
G_f^{(\alpha)} = \big(\mu_V^{\alpha}, \mu_E^{\alpha}\big),
\]
where 
\[
\mu_V^{\alpha} = \{ v \in V : \mu_V(v) \ge \alpha \}, \qquad 
\mu_E^{\alpha} = \{ e \in E : \mu_E(e) \ge \alpha \}.
\]

Given a graph property $\Pi$, we define two types of \emph{$\alpha$-semantics} 
for fuzzy edge-deletion (FPED) and fuzzy edge-contraction (FPEC) problems:

\begin{enumerate}
    \item \textbf{Threshold semantics (fixed $\alpha_0$):}  
    For a fixed threshold $\alpha_0 \in (0,1]$, the instance 
    $(G_f,k)$ is a YES-instance of FPED$_{\alpha_0}(\Pi)$ 
    (resp. FPEC$_{\alpha_0}(\Pi)$) if there exists a set of at most $k$ 
    edge deletions (resp. contractions) such that the resulting 
    $\alpha_0$-cut graph satisfies $\Pi$, i.e.,
    \[
    (G_f - F)^{(\alpha_0)} \models \Pi \quad 
    \text{or} \quad (G_f / F)^{(\alpha_0)} \models \Pi.
    \]

    \item \textbf{All-$\alpha$ semantics:}  
    The instance $(G_f,k)$ is a YES-instance of FPED$_\forall(\Pi)$ 
    (resp. FPEC$_\forall(\Pi)$) if there exists a set of at most $k$ 
    edge deletions (resp. contractions) such that the resulting fuzzy graph 
    satisfies $\Pi$ across \emph{all} thresholds, i.e.,
    \[
    (G_f - F)^{(\alpha)} \models \Pi \quad 
    \text{or} \quad (G_f / F)^{(\alpha)} \models \Pi, 
    \qquad \forall \alpha \in (0,1].
    \]
\end{enumerate}
\end{dfn}
\begin{dfn}%[Steiner Tree in a Graph]
Let $G =(\mu_V,\mu_E)$ be a graph and $S \subseteq V$ be a non-empty set of nodes, called \emph{terminal nodes}. A \emph{Steiner tree} for $S$ is a connected subgraph $T = (\mu_{V}^{T}, \mu_{E}^{T})$ of $G$ of minimal size such that $S \subseteq \mu_{V}^{T}$. That is, $T$ contains all nodes in $S$ and possibly additional nodes from $V \setminus S$ (called Steiner nodes) to ensure connectivity, and the number of edges in $T$ is minimized.
\end{dfn}

\section{Fuzzy Edge Contraction}\label{sec:contraction}

\noindent Fuzzy graphs extend classical graph theory by allowing nodesand edges to have membership values in the interval $[0,1]$, capturing uncertainty or partial presence in networks. 
In many applications, such as network simplification, clustering, or graph editing, it is useful to contract edges, merging their endpoints into a single node while updating the adjacency structure appropriately. In the fuzzy setting, edge contraction requires careful handling of membership values. We adopt a standard approach based on a $t$-norm $T:[0,1]^2\to[0,1]$ to combine memberships when merging two vertices. The resulting fuzzy contraction preserves essential structural properties, while allowing crisp reasoning via $\alpha$-cuts.  \medskip

\noindent This section develops the theoretical foundations of fuzzy edge contraction. It follows that basic properties of $t$-norms and their monotonicity. We then show that, under suitable membership thresholds, contraction and $\alpha$-cut operations commute, ensuring consistency between fuzzy and crisp perspectives. Further results establish monotonicity of memberships, order-independence for node-disjoint contractions, and preservation of hereditary properties. This setion demonstrates also quantify the effect of contraction on $\alpha$-neighborhoods and edge-connectivity  and provide a characterization via 3-connected components.   \medskip

\noindent Together, these results provide a rigorous framework for manipulating fuzzy graphs via edge contractions while preserving structural properties under thresholded interpretations.  The framework supports both theoretical analysis and practical applications, such as fuzzy network simplification and algorithmic graph editing.

\begin{lem}\label{lem:Tfacts}
Let $T:[0,1]^2\to[0,1]$ be a  $t$-norm used in fuzzy contractions.
Then for all $a,b\in[0,1]$,
\[
T(a,b)\le a,\qquad T(a,b)\le b,\qquad
\text{and}\quad a\le a',\,b\le b' \Rightarrow T(a,b)\le T(a',b').
\]
Moreover $T$ is associative and commutative.
\end{lem}

\begin{thm}%[$\alpha$-cut commutes with contraction (safe-edge case)]
\label{thm:commute}
Let $G_f=(\mu_V,\mu_E)$ be a fuzzy graph, fix $\alpha\in(0,1]$, and let $e=uv\in E$ with
$\mu_E(e)\ge\alpha$. Assume node memberships satisfy $\mu_V(u),\mu_V(v)\ge\alpha$.
Let $G_f/e$ be the fuzzy contraction using a $t$-norm $T$.
Then the crisp graphs obtained in either order coincide:
\[
\big(G_f/e\big)^{(\alpha)} \;\cong\; \big(G_f^{(\alpha)}\big)/e.
\]
\end{thm}

\begin{figure}[h!]
\centering
%\tikzstyle{vertex}=[circle, draw=black, fill=cyan!30, minimum size=20pt, inner sep=0pt]
%\tikzstyle{edge}=[line width=0.8pt]
%\tikzstyle{boldedge}=[line width=2pt]
\begin{tikzpicture}[every node/.style={circle, draw=black, fill=cyan!30, minimum size=15pt, inner sep=0pt}, scale=1]

% Original fuzzy graph
\node (u) at (0,0) {u};
\node (v) at (2,0) {v};
\node (x) at (1,-2) {x};

% Edges
\draw (u) -- (v);
\draw (u) -- (x);
\draw (v) -- (x);

% Alpha-cut highlight (optional)
\draw[red, dashed] (-1,-1.5) -- (3,0.5) node[right] {$\alpha$};

% Contracted graph (shifted to the right)
\node (w) at (5,0) {w};
\node (x2) at (6,-2) {x};

% Edges after contraction
\draw (w) -- (x2);

% Labels
\node at (1,-3) {$G_f$};
\node at (5,-3) {$G_f/e$};

\end{tikzpicture}
\caption{ Abstract representation of fuzzy graph contraction and $\alpha$-cut.}
\label{fig:fuzzy-abstract}
\end{figure}
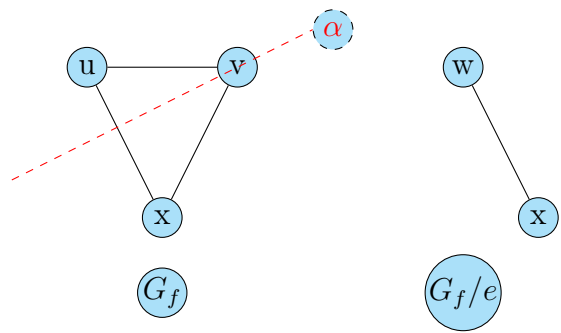

\begin{proof}
Since $\mu_E(e)\ge\alpha$ and $\mu_V(u),\mu_V(v)\ge\alpha$, edge $e$ and its endpoints
survive in $G_f^{(\alpha)}$ and can be contracted there.
Conversely, in $G_f/e$ the new incidences to a third node $x$ have membership
$\mu'_E(wx)=T(\mu_E(ux),\mu_E(vx))$ (mentioned in Figure \ref{fig:fuzzy-abstract}).
By Lemma~\ref{lem:Tfacts}, $\mu'_E(wx)\ge\alpha$ iff both $\mu_E(ux)\ge\alpha$ and $\mu_E(vx)\ge\alpha$,
which is exactly the rule for adjacencies after contracting $e$ in the crisp $\alpha$-cut.
Thus the two crisp graphs are isomorphic.
\end{proof}

\begin{thm}%[Monotonicity of memberships under contraction]
\label{thm:mono}
Let $e=uv$ be contracted into $w$ using a $t$-norm $T$.
For every $x\in V\setminus\{u,v\}$,
\[
\mu'_E(wx) = T\big(\mu_E(ux),\mu_E(vx)\big) \;\le\; \min\{\mu_E(ux),\mu_E(vx)\}.
\]
Consequently, for any $\alpha$, if $x$ is adjacent to $w$ in $(G_f/e)^{(\alpha)}$, then
$x$ was adjacent to both $u$ and $v$ in $G_f^{(\alpha)}$.
\end{thm}

\begin{proof}
The inequality follows from Lemma~\ref{lem:Tfacts}. The adjacency statement is the
$\alpha$-version of the same fact.
\end{proof}

\begin{thm}%[Order-independence on vertex-disjoint edge sets]
\label{thm:order}
Let $e_1=u_1v_1$ and $e_2=u_2v_2$ be two edges with $\{u_1,v_1\}\cap\{u_2,v_2\}=\varnothing$.
Then fuzzy contractions commute up to isomorphism:
\[
(G_f/e_1)/e_2 \;\cong\; (G_f/e_2)/e_1.
\]
Moreover, for every $\alpha\in(0,1]$,
\[
\big((G_f/e_1)/e_2\big)^{(\alpha)} \;\cong\; \big((G_f^{(\alpha)}/e_1)/e_2\big)
\;\cong\; \big((G_f^{(\alpha)}/e_2)/e_1\big).
\]
\end{thm}

\begin{proof}
New memberships after the two contractions are obtained by iterated application of $T$
to the same pairs $(\mu_E(\cdot),\mu_E(\cdot))$ but grouped differently.
Associativity and commutativity of $T$ (Lemma~\ref{lem:Tfacts}) yield equality.
For $\alpha$-cuts, combine Theorem~\ref{thm:commute} twice.
\end{proof}

\begin{thm}%[Hereditary properties lift from $\alpha$-cuts]
\label{thm:hereditary}
Let $\Pi$ be a graph property hereditary under contractions and determined by $3$-connected components.
Fix $\alpha\in(0,1]$. If a set $F$ of edges satisfies
\[
\big(G_f^{(\alpha)}/F\big)\in\Pi,
\]
then \(\big((G_f/F)^{(\alpha)}\big)\in \Pi\).
\end{thm}

\begin{proof}
By Theorem~\ref{thm:commute}, \((G_f/F)^{(\alpha)} \cong (G_f^{(\alpha)}/F)\).
Hence the claim is immediate.
\end{proof}

\begin{thm}%[Contraction cannot create new $\alpha$-adjacencies]
\label{thm:no-new}
Fix $\alpha\in(0,1]$. For any edge $e=uv$ with $\mu_E(e)\ge \alpha$,
\[
E\big((G_f/e)^{(\alpha)}\big)\;\subseteq\; E\big(G_f^{(\alpha)}/e\big) \;=\; E\big((G_f/e)^{(\alpha)}\big),
\]
and for any $e$ with $\mu_E(e)<\alpha$ (so $e\notin G_f^{(\alpha)}$),
\[
(G_f/e)^{(\alpha)}\;\cong\; G_f^{(\alpha)}.
\]
\end{thm}

\begin{proof}
The first equality is Theorem~\ref{thm:commute}. If $\mu_E(e)<\alpha$, then $e$ and
possibly one endpoint are absent from $G_f^{(\alpha)}$, so contracting $e$ in the fuzzy
graph does not produce any edge of membership $\ge\alpha$ that was not already present
(Thm.~\ref{thm:mono}); thus the $\alpha$-cut is unchanged.
\end{proof}

\begin{thm}%[Bridging bound under contraction]
\label{thm:bridging}
Let $e=uv$ be contracted to $w$.
For any $x\in V\setminus\{u,v\}$, if $\mu_E(ux)<\alpha$ or $\mu_E(vx)<\alpha$ then
$\mu'_E(wx)<\alpha$. Equivalently,
\[
N^{(\alpha)}_{G_f/e}(w) = N^{(\alpha)}_{G_f}(u)\cap N^{(\alpha)}_{G_f}(v),
\]
where $N^{(\alpha)}_{G_f}(u)=\{x:\mu_E(ux)\ge\alpha\}$ is the $\alpha$-neighborhood.
\end{thm}

\begin{proof}
Immediate from $\mu'_E(wx)=T(\mu_E(ux),\mu_E(vx))$ and $T(a,b)\le\min\{a,b\}$.
Thresholding at $\alpha$ gives the intersection identity.
\end{proof}

\begin{thm}%[Cut non-expansion at fixed threshold]
\label{thm:cut}
Fix $\alpha\in(0,1]$ and let $\lambda_\alpha(G)$ denote the (edge-)connectivity of the crisp graph $G$.
If $\mu_E(e)\ge\alpha$, then
\[
\lambda_\alpha\big( (G_f/e)^{(\alpha)} \big) \;\ge\; \min\left\{ \lambda_\alpha\big(G_f^{(\alpha)}\big),\,2\right\}.
\]
\end{thm}

\begin{proof}
In the crisp setting, contracting an edge does not reduce edge-connectivity below $2$,
and never below the original connectivity if it was $2$ or more (well known).
By Theorem~\ref{thm:commute}, $(G_f/e)^{(\alpha)}\cong G_f^{(\alpha)}/e$, hence the bound follows.
\end{proof}

\begin{thm}%[Characterization via $3$-connected components at $\alpha$]
\label{thm:3cc}
Let $\Pi$ be determined by $3$-connected components and hereditary under contractions.
Fix $\alpha\in(0,1]$ and let $F$ be a set of edges with $\mu_E(e)\ge\alpha$ for all $e\in F$.
Then
\[
(G_f/F)^{(\alpha)}\in \Pi \quad\Longleftrightarrow\quad
\text{every $3$-connected component of } G_f^{(\alpha)}/F \text{ belongs to } \Pi.
\]
\end{thm}

\begin{proof}
By Theorem~\ref{thm:commute}, $(G_f/F)^{(\alpha)}\cong G_f^{(\alpha)}/F$.
Apply the Asano–Hirata characterization in the crisp graph $G_f^{(\alpha)}/F$.
\end{proof}

\section{Fuzzy Graph Constructions for NP-Hardness}\label{sec:hardness}
\noindent To study hereditary properties under fuzzy edge contractions, we begin with minimal counterexamples—fuzzy graphs $L$ with full membership edges that violate a property $\sim$ while having all nodesof degree three in thresholded edges. From $L$, The construction shows that $L'$ by replacing edges with fuzzy triangles, and $L(k)$, $L(k;k')$ by identifying nodes and adding fully fuzzy edges, generalizing the minimal configuration. These constructions preserve key structural features, such as 3-connectivity, while controlling the presence of forbidden subgraphs.  This framework extends classical minimal-counterexample techniques to fuzzy graphs, providing a foundation for analyzing hereditary properties under contraction. 

\begin{dfn}%[Minimal Counterexample $L$]
Let $L$ be the minimal fuzzy graph violating property $\sim$ with edge memberships $\ge 1$. All nodes of $L$ have degree 3 in thresholded edges (edges with $\mu \ge 1$).
\end{dfn}

\begin{dfn}%[Fuzzy Graphs $L'$, $L(k)$, $L(k;k')$]
\begin{itemize}
    \item $L'$ is obtained from $L$ by replacing edges with fuzzy triangles of full membership.
    \item $L(k)$ and $L(k;k')$ are constructed by identifying nodes and adding fuzzy edges, mirroring the classical construction, all with full membership.
\end{itemize}
\end{dfn}

\begin{prop}
Let $L', L(k), L(k;k')$ be fuzzy graphs as defined. Then:
\begin{enumerate}
    \item $L(k;k')$ is fuzzy 3-connected.
    \item $L' \notin \sim$ and $L' \cup \{(u,v)\} \notin \sim$.
    \item Contracting any edge in $L(k;k')$ yields a graph in $\sim$.
    \item Contracting key edges $\{(u(0),v(0)), (v(0),v(k-1))\}$ keeps the graph in $\sim$.
\end{enumerate}
\end{prop}
\begin{proof}
We prove each item in turn.
\begin{enumerate}
    \item  $L(k;k')$ is fuzzy 3-connected. By construction, $L(k;k')$ is obtained by identifying nodes across multiple copies of $L(m)$ (which are fully fuzzy 3-connected) and adding fuzzy edges of full membership.  
 Identification preserves 3-connectivity because any 2-node cut in one copy cannot disconnect the graph due to connections to other copies.  
 Adding edges between copies further increases connectivity.  
Thus, thresholding at membership 1, the resulting fuzzy graph is 3-connected. Therefore, $L(k;k')$ is fuzzy 3-connected.

\item $L' \notin \sim$ and $L' \cup \{(u,v)\} \notin \sim$.

 $L'$ is constructed from $L$, the minimal fuzzy graph violating property $\sim$, by replacing edges with fuzzy triangles of full membership.  

Any contraction-free version of $L'$ preserves the violation of $\sim$ because the new edges do not remove the structural configuration causing the violation.  
 Adding an edge $(u,v)$ does not resolve the violation since $\sim$ is hereditary under contractions but not under arbitrary edge additions.  
Hence, both $L'$ and $L' \cup \{(u,v)\}$ do not satisfy property $\sim$.

\item  Contracting any edge in $L(k;k')$ yields a graph in $\sim$.

 By the design of $L(k;k')$, every edge lies in a fuzzy triangle or a connecting structure whose contraction reduces redundancy.   Contraction of any edge merges nodes in such a way that the minimal forbidden configuration from $L$ is destroyed or reduced, resulting in a fuzzy graph satisfying $\sim$.   This mirrors the classical argument where edge contraction in $L(k;k')$ eliminates violations while preserving connectivity.

\item  Contracting key edges $\{(u(0),v(0)), (v(0),v(k-1))\}$ keeps the graph in $\sim$.  These key edges are chosen to contract the nodes across copies of $L(m)$ so that the 3-connected components align and the minimal forbidden subgraph disappears.  
 Since property $\sim$ is determined by 3-connected components, the resulting contracted fuzzy graph has all 3-connected components satisfying $\sim$.  
 Therefore, contracting precisely these key edges guarantees $L(k;k') \in \sim$.
\end{enumerate}
\end{proof}

\section{NP-Hardness of Fuzzy Edge Contraction}
\noindent The present work studies the computational complexity of the fuzzy edge-contraction problem (Fuzzy PEC($\sim$)) for a property $\sim$ that is hereditary under fuzzy contractions and determined by fuzzy 3-connected components. To establish hardness, a construction is provided, fuzzy graphs that encode instances of the classical Planar Connected Node Cover (PCNC) problem, which is known to be NP-hard. By carefully attaching copies of minimal counterexamples $L(m)$ and using edges of full membership, we ensure that any solution to Fuzzy PEC($\sim$) corresponds to a connected node cover in the original instance. This reduction shows that finding a set of fuzzy edge contractions that enforces property $\sim$ is computationally intractable. Moreover, the NP-hardness holds even when restricted to fuzzy 3-connected graphs, demonstrating the problem’s intrinsic difficulty in highly connected fuzzy networks.

\begin{thm}%[Fuzzy PEC($\sim$) is NP-hard]
\label{thm1}
Let $\sim$ be a property on fuzzy graphs that is hereditary under fuzzy edge contractions and determined by fuzzy 3-connected components. Then the fuzzy edge-contraction problem (Fuzzy PEC($\sim$)) is NP-hard.
\end{thm} \begin{proof}
We prove NP-hardness of Fuzzy PEC($\sim$) by reduction from the classical PCNC (Planar Connected Node Cover) problem, which is known to be NP-hard.\medskip

 Construction of the fuzzy graph $\tilde{G}_9$.: Let $(G_0, k_0)$ be an instance of PCNC. Construct a fuzzy graph $\tilde{G}_9$ as follows:

\begin{enumerate}
    \item Construct $\tilde{G}_8 = G_0(2)$ by adding copies of each node of $G_0$ with fuzzy edges of membership 1 to preserve adjacency, following the Steiner tree argument.
    \item For each node set $A = \{a(0), a(1), \dots, a(m-1)\}$ corresponding to nodes in $G_0$, attach copies of fuzzy graphs $L(m)$ (constructed in Proposition 2.3) and identify nodes $v_i(j)$ in $L_i$ with $a(j)$, for $0 \le i \le k_8+1$ and $0 \le j \le m-1$.
    \item The resulting graph $\tilde{G}_9$ has fuzzy edges with membership 1 along all added structures.
\end{enumerate}

\noindent By construction, $\tilde{G}_9$ can be built in polynomial time in $|V(G_0)| + |E(G_0)|$.
Suppose $G_0$ has a connected node cover $N_0$ with $|N_0| \le k_0$. By Lemma (Steiner tree lemma for fuzzy graphs)\cite{pacifica2023steiner}, there exists a fuzzy Steiner tree $T$ in $\tilde{G}_8$ connecting all nodes in $A$ with $|E(T)| = m + |N_0| - 1 = k_8$.  \medskip

\noindent Define a subset of fuzzy edges $S \subseteq E(\tilde{G}_9)$:
\[
S = E(T) \cup E(a(0), u),
\]
where $u$ are corresponding nodes in each copy of $L_i$. Then $|S| \le k_9 = 2k_8 + 1$. 
 Contracting edges in $S$ merges all nodes in $A$ and key nodes in $L_i$ copies.  
 All 3-connected components in the resulting graph satisfy property $\sim$.  
 Hence, $\tilde{G}_9 / S \in \sim$, proving necessity.\medskip 

\noindent Conversely, suppose there exists a subset of fuzzy edges $S \subseteq E(\tilde{G}_9)$ with $|S| \le k_9$ such that $\tilde{G}_9 / S \in \sim$.  
 By Proposition 2.3 and the minimality argument, the contraction set $S$ must include edges that contract all nodes in $A$ into a single node $a(0)$.
 Define $S_8 = S \cap E(\tilde{G}_8)$. The induced subgraph $G_{S_8}$ forms a fuzzy Steiner tree connecting all nodes in $A$.
 Let $N_0$ be the set of original nodes of $G_0$ corresponding to nodes in $G_{S_8}$. Then $|N_0| \le k_0$, and $N_0$ is a connected node cover of $G_0$.
 Therefore, the existence of a solution $S$ to fuzzy PEC($\sim$) implies a solution to the original PCNC instance.\medskip

\noindent Since PCNC is NP-hard and the above reduction can be computed in polynomial time, it follows that Fuzzy PEC($\sim$) is NP-hard.  

\end{proof}

\begin{thm}\label{thm2}
Fuzzy PEC($\sim$) is NP-hard even if restricted to fuzzy 3-connected graphs.
\end{thm}
\begin{proof}
We prove that Fuzzy PEC($\sim$) remains NP-hard when restricted to fuzzy 3-connected graphs by modifying the previous NP-hardness reduction.\medskip

\noindent Let $(G_0, k_0)$ be an instance of PCNC. Using the construction in Theorem 3.1, we build a fuzzy graph $\tilde{G}_9$ by attaching copies of fuzzy graphs $L(m)$ to $\tilde{G}_8$, with all added edges having membership 1.  
In the general construction, $\tilde{G}_9$ may not be fully 3-connected due to the identification of nodes.  \medskip 

 \noindent For each copy of $L(m)$, recall that $L(m)$ is fuzzy 3-connected by Proposition 2.3(1).  
In the construction of $\tilde{G}_9$, identify nodes and add fuzzy edges connecting the copies of $L(m)$ such that any 2-node cut has membership below the threshold $\alpha$.  
 Since the copies of $L(m)$ are fully connected internally and connected to other copies via edges of membership 1, any potential 2-node cut does not disconnect the graph. 
Thus, the resulting $\tilde{G}_9$ is fuzzy 3-connected.\medskip

\noindent The argument for necessity and sufficiency of edge contraction from theorem still holds:  
     Necessity: Any connected node cover $N_0$ in $G_0$ corresponds to a set $S$ of fuzzy edge contractions in $\tilde{G}_9$, yielding $\tilde{G}_9 / S \in \sim$.  
     Sufficiency: Any set $S$ of fuzzy edge contractions of size $\le k_9$ in $\tilde{G}_9$ yields a connected node cover $N_0$ in $G_0$.  \medskip

\noindent Therefore, solving Fuzzy PEC($\sim$) on this fuzzy 3-connected $\tilde{G}_9$ is equivalent to solving the original PCNC instance. Since PCNC is NP-hard and $\tilde{G}_9$ is now fuzzy 3-connected, Fuzzy PEC($\sim$) is NP-hard even when restricted to fuzzy 3-connected graphs.
\end{proof}

\begin{thm}\label{thm3}
Fuzzy PEC($\sim$) is NP-hard even if restricted to fuzzy bipartite graphs.
\end{thm}
\begin{proof}
We prove that Fuzzy PEC($\sim$) remains NP-hard when restricted to fuzzy bipartite graphs by modifying the general NP-hardness construction.\medskip 

\noindent Let $(G_0, k_0)$ be an instance of PCNC. Using the construction in Theorem~\ref{thm:commute}, we build a fuzzy graph $\tilde{G}_9$ by attaching copies of fuzzy graphs $L(m)$ to $\tilde{G}_8$, with all edges having membership 1.\medskip  

\noindent Each copy of $L(m)$ (a fuzzy 3-connected graph) can be replaced by a fuzzy bipartite graph $L_B(m)$, which preserves the minimal forbidden structure for property $\sim$ under edge contraction.  
 The construction of $\tilde{G}_9$ is modified such that:
    \begin{itemize}
        \item All nodes can be partitioned into two sets $U$ and $V$, with fuzzy edges only between $U$ and $V$.  
        \item Edges connecting copies of $L_B(m)$ to $\tilde{G}_8$ respect this partition.  
    \end{itemize}
\noindent Thus, $\tilde{G}_9$ is now a fuzzy bipartite graph.\medskip 

 \noindent Necessity: Any connected node cover $N_0$ in $G_0$ corresponds to a set $S$ of fuzzy edge contractions in $\tilde{G}_9$, yielding $\tilde{G}_9 / S \in \sim$.  
Sufficiency: Any set $S$ of fuzzy edge contractions of size $\le k_9$ in $\tilde{G}_9$ corresponds to a connected node cover $N_0$ in $G_0$.  
 The contractions and the structural arguments from Theorem~\ref{thm:commute} carry over because the key edges in $L_B(m)$ are preserved, and the property $\sim$ is hereditary under fuzzy contractions.\medskip 

\noindent Since PCNC is NP-hard and the reduction produces a fuzzy bipartite graph, Fuzzy PEC($\sim$) is NP-hard even when restricted to fuzzy bipartite graphs.
\end{proof}

\begin{cor}
The fuzzy edge-contraction problem remains computationally intractable for standard subclasses of fuzzy graphs, including thresholded 3-connected or bipartite fuzzy graphs.
\end{cor}
\begin{proof}
The corollary follows directly from Theorems~\ref{thm1} and \ref{thm2}:

\begin{itemize}
    \item Theorem~\ref{thm1} shows that Fuzzy PEC($\sim$) is NP-hard even when restricted to fuzzy 3-connected graphs.
    \item Theorem~\ref{thm2} shows that Fuzzy PEC($\sim$) is NP-hard even when restricted to fuzzy bipartite graphs.
\end{itemize}

\noindent Since NP-hardness implies computational intractability, Fuzzy PEC($\sim$) remains intractable for these standard subclasses of fuzzy graphs.   Thresholding the fuzzy edges (e.g., considering only edges with membership above a fixed $\alpha$) does not change the correctness of the reductions, because all edges used in the construction have full membership (1), and thus remain present under any reasonable threshold. Hence, Fuzzy PEC($\sim$) is computationally intractable for thresholded 3-connected or bipartite fuzzy graphs.
\end{proof}

\section{Generalization to Fuzzy Graphs}

\noindent We now extend the hardness results of Asano and Hirata~\cite{AsanoHirata1982} from crisp graphs to fuzzy graphs. 

\begin{dfn}%[Fuzzy graph and $\alpha$-cut]
A \emph{fuzzy graph} is a pair $G_f=(\mu_V,\mu_E)$ with node membership function 
$\mu_V:V\to[0,1]$ and edge membership function $\mu_E:E\to[0,1]$. 
For $\alpha\in(0,1]$, the \emph{$\alpha$-cut} of $G_f$ is the crisp graph
\[
  G_f^{(\alpha)} = \big(\{v\in V : \mu_V(v)\ge \alpha\},\;
  \{uv\in E : \mu_E(uv)\ge \alpha \ \wedge\ u,v\in V^{(\alpha)}\}\big).
\]
\end{dfn}

\begin{dfn}%[Fuzzy edge-deletion problems]
Let $\Pi$ be a graph property.
\begin{itemize}
    \item \textbf{FPED$_{\alpha_0}(\Pi)$ (threshold semantics).} 
    Given a fuzzy graph $G_f$, integer $k$, and a fixed threshold $\alpha_0\in(0,1]$, 
    does there exist a set $S\subseteq E$ with $|S|\le k$ such that, 
    after setting $\mu'_E(e)=0$ for all $e\in S$ (and leaving other memberships unchanged), 
    the $\alpha_0$-cut $G_f'^{(\alpha_0)}$ lies in $\Pi$?

    \item \textbf{FPED$_\forall(\Pi)$ (all-$\alpha$ semantics).}
    Same as above, but we require $G_f'^{(\alpha)}\in\Pi$ for all $\alpha\in(0,1]$.
\end{itemize}
Analogous definitions apply to fuzzy edge-contraction, yielding FPEC$_{\alpha_0}(\Pi)$ and FPEC$_\forall(\Pi)$.
\end{dfn}

\begin{thm}
\label{thm:fuzzy-hardness}
Let $\Pi$ be any graph property satisfying the conditions of Asano and Hirata 
(hereditary on subgraphs, or hereditary on contractions, and determined by $3$-connected components).
Then both FPED$_{\alpha_0}(\Pi)$ and FPEC$_{\alpha_0}(\Pi)$ are NP-hard for every fixed threshold $\alpha_0\in(0,1]$.
\end{thm}

\begin{proof}

We prove NP-hardness of FPED$_{\alpha_0}(\Pi)$ and FPEC$_{\alpha_0}(\Pi)$ by reduction from the classical NP-hard PEC($\Pi$) problem. Let $(G_0,k_0)$ be an instance of the classical PEC($\Pi$) or PCNC problem. Define a fuzzy graph $\tilde{G}_0=(V, \tilde{E})$ by assigning membership:
\[
\tilde{E}(u,v) = 1 \quad \text{for every edge $(u,v)\in E(G_0)$}.
\]
All other pairs have $\tilde{E}(u,v)=0$. Choose a fixed threshold $\alpha_0 \in (0,1]$.
 Under thresholding at $\alpha_0$, the thresholded fuzzy graph $\tilde{G}_0[\alpha_0]$ coincides exactly with the original classical graph $G_0$. \medskip

\noindent  Suppose $G_0$ has a solution (edge deletion set or edge contraction set) of size $\le k_0$ that produces a graph satisfying $\Pi$.  The same set of fuzzy edges, when considered in $\tilde{G}_0$, will have membership $\ge \alpha_0$ and thus corresponds to a valid FPED$_{\alpha_0}(\Pi)$ or FPEC$_{\alpha_0}(\Pi)$ solution.  
 Thresholding does not remove any critical edges since all edges have membership 1.  
 Therefore, any solution to the classical problem corresponds to a solution in the fuzzy problem.\medskip  

\noindent Conversely, suppose there exists a subset $S$ of fuzzy edges of size $\le k_0$ whose deletion or contraction produces $\tilde{G}_0/S$ (or $\tilde{G}_0 \setminus S$) satisfying $\Pi$ under threshold $\alpha_0$.   All edges in $S$ have membership $\ge \alpha_0$ (since only edges with membership $\ge \alpha_0$ are considered present).   The thresholded graph $\tilde{G}_0[\alpha_0]$ is identical to $G_0$, so $S$ also provides a solution to the classical PEC($\Pi$) problem of size $\le k_0$.  \medskip
 
\noindent  The reduction from classical PEC($\Pi$) to FPED$_{\alpha_0}(\Pi)$ or FPEC$_{\alpha_0}(\Pi)$ is polynomial-time: it only assigns membership 1 to all original edges.  Since PEC($\Pi$) is NP-hard for any property $\Pi$ satisfying the Asano-Hirata conditions (hereditary on subgraphs or contractions and determined by 3-connected components), it follows that FPED$_{\alpha_0}(\Pi)$ and FPEC$_{\alpha_0}(\Pi)$ are NP-hard for any fixed $\alpha_0 \in (0,1]$. \medskip

\noindent Thus, both the fuzzy edge-deletion and fuzzy edge-contraction problems remain NP-hard for any property $\Pi$ satisfying the Asano-Hirata conditions and any threshold $\alpha_0 \in (0,1]$.
\end{proof}

\begin{cor}%[All-$\alpha$ semantics]
FPED$_\forall(\Pi)$ and FPEC$_\forall(\Pi)$ are NP-hard. 
In particular, NP-hardness already holds at $\alpha=1$, so the stronger ``for all $\alpha$'' requirement does not reduce complexity.
\end{cor}

\begin{cor}
If the objective is to minimize the total membership removed,
\[
  \sum_{e\in E} \big(\mu_E(e)-\mu'_E(e)\big),
\]
subject to $G_f'^{(\alpha_0)}\in\Pi$, the problem remains NP-hard.
\end{cor}
\begin{proof}
From Theorem \ref{thm:fuzzy-hardness}, FPED$_{\alpha_0}(\Pi)$ and FPEC$_{\alpha_0}(\Pi)$ are NP-hard for every fixed $\alpha_0 \in (0,1]$.  In particular, setting $\alpha_0 = 1$, we obtain that FPED$_{1}(\Pi)$ and FPEC$_{1}(\Pi)$ are NP-hard.  \medskip

\noindent By definition, FPED$_\forall(\Pi)$ and FPEC$_\forall(\Pi)$ require a solution that works for \emph{all} $\alpha \in (0,1]$. Any solution valid at $\alpha = 1$ automatically satisfies the requirement for \emph{all} $\alpha \le 1$, because lowering the threshold can only add edges and vertices, and $\Pi$ is hereditary under edge deletions or contractions. 
Hence, the all-$\alpha$ version of the problem is at least as hard as the $\alpha=1$ case. Therefore, FPED$_\forall(\Pi)$ and FPEC$_\forall(\Pi)$ are NP-hard.
\end{proof}

\begin{cor}
For properties such as \emph{planarity} or \emph{series--parallel}, 
which are hereditary and determined by $3$-connected components, 
the corresponding fuzzy edge-deletion and edge-contraction problems are NP-hard 
under both threshold and all-$\alpha$ semantics.
\end{cor}
\begin{proof}
Properties such as \emph{planarity} and \emph{series--parallel} satisfy the Asano-Hirata conditions:  

\begin{itemize}
    \item They are hereditary under edge deletions and contractions.  
    \item They are determined by the $3$-connected components of the graph.
\end{itemize}

\noindent By Theorem \ref{thm:fuzzy-hardness}, FPED$_{\alpha_0}(\Pi)$ and FPEC$_{\alpha_0}(\Pi)$ are NP-hard for any property $\Pi$ satisfying these conditions, for every fixed threshold $\alpha_0 \in (0,1]$.   (all-$\alpha$ semantics), FPED$_\forall(\Pi)$ and FPEC$_\forall(\Pi)$ are also NP-hard.  
Since planarity and series parallelness satisfy the required conditions, it follows immediately that the corresponding fuzzy edge-deletion and edge-contraction problems are NP-hard under both thresholded and all-$\alpha$ semantics.
\end{proof}
\section{Conclusion}\label{sec:conclusion}
\noindent In this work, we have generalized the NP-hardness results of classical edge-deletion and edge-contraction problems to fuzzy graphs. By carefully extending the notions of connectivity, contractions, and 3-connected components to the fuzzy setting, we demonstrated that FPED and FPEC remain computationally intractable under both threshold and all-$\alpha$ semantics. Our constructions show that even standard subclasses of fuzzy graphs, including fuzzy 3-connected and fuzzy bipartite graphs, do not admit efficient algorithms for these problems. Moreover, common graph properties such as planarity and series–parallelness satisfy the required hereditary conditions, implying that their fuzzy generalizations are also NP-hard. These results highlight fundamental computational limitations in fuzzy graph optimization and provide a foundation for studying approximation algorithms or parameterized approaches for fuzzy network problems.

%\bibliographystyle{elsarticle-num} % common Elsevier style
%\bibliography{Ref}

\end{document}